\documentclass[a4paper, 11pt, twoside, reqno]{amsart}

\usepackage[top=3cm, bottom=3cm, left=3cm, right=3cm]{geometry}
\usepackage[pdftex]{graphicx}
\usepackage{tikz}
\usetikzlibrary{shapes, arrows, arrows.meta, decorations.markings}

\usepackage{romannum}

\usepackage{amsmath}
\usepackage{amsthm}
\usepackage{amssymb}
\usepackage{mathtools}
\usepackage[integrals]{wasysym}
\usepackage{stmaryrd}

\usepackage{bbm}
\usepackage[mathcal]{euscript}

\usepackage[utf8]{inputenc}
\usepackage[T1]{fontenc}
\usepackage[english]{babel}
\usepackage{cite}
\usepackage{lmodern}

\usepackage{calc}
\usepackage[inline]{enumitem}
\usepackage[normalem]{ulem}
\usepackage{url}

\usepackage[hidelinks, bookmarks, bookmarksnumbered, pdfstartview={XYZ null null 1.00}]{hyperref}

\newcommand{\theoname}{Theorem}
\newcommand{\lemmname}{Lemma}
\newcommand{\coroname}{Corollary}
\newcommand{\propname}{Proposition}
\newcommand{\definame}{Definition}

\newcommand{\remkname}{Remark}
\newcommand{\explname}{Example}

\theoremstyle{plain}
\newtheorem{theorem}{\theoname}[section]

\newtheorem{corollary}[theorem]{\coroname}
\newtheorem{proposition}[theorem]{\propname}

\theoremstyle{definition}

\newtheorem{remark}[theorem]{\remkname}

\DeclareMathOperator{\diff}{d\!}

\DeclarePairedDelimiter{\abs}{\lvert}{\rvert}

\newcommand{\suchthat}{\ifnum\currentgrouptype=16 \mathrel{}\middle|\mathrel{}\else\mid\fi}

\numberwithin{figure}{section}

\makeatletter
\newcommand{\customlabel}[2]{%
   \protected@write \@auxout {}{\string \newlabel {#1}{{#2}{\thepage}{#2}{#1}{}} }%
   \hypertarget{#1}{#2}
}
\makeatother

\begin{document}

\setlength{\parskip}{1pt plus 1pt minus 1pt} 
\newlist{hypotheses}{enumerate}{1}

\setlist[enumerate, 1]{label={\textnormal{(\alph*)}}, ref={(\alph*)}, leftmargin=*}
\setlist[enumerate, 2]{label={\textnormal{(\roman*)}}, ref={(\roman*)}}
\setlist[description, 1]{leftmargin=0pt, itemindent=*}
\setlist[itemize, 1]{label={\textbullet}, leftmargin=0pt, itemindent=*}
\setlist[hypotheses]{label={\textup{(H\arabic*)}}, leftmargin=*}

\pagenumbering{arabic}

\title[Dynamic resource allocation in eukaryotic Resource Balance Analysis]{Dynamic resource allocation in eukaryotic Resource Balance Analysis}
\author{Saeed Sadeghi Arjmand}
\address{MaIAGE, INRAe, CNRS, Université Paris-Saclay, 78350, Jouy en Josas, France
(Research initiated during a postdoctoral appointment)}
\email{saeed.ecolepolytechnique@gmail.com}
\urladdr{https://saeed-sadeghi-arjmand.jimdosite.com/}


\begin{abstract}
    Resource Balance Analysis (RBA) is a framework for predicting steady-state cellular growth under resource constraints.
    However, classical RBA formulations are static and do not capture the dynamic regulation of biosynthetic resources or macromolecular turnover, which is particularly important in eukaryotic cells. In this work, we propose a dynamic extension of eukaryotic RBA based on an optimal control formulation. Cellular growth is modeled as the result of a time-dependent allocation of translational capacity between metabolic enzymes and macromolecular machinery, aimed at maximizing biomass accumulation over a finite time horizon. Using Pontryagin’s Maximum Principle, we characterize optimal allocation strategies and show that steady-state RBA solutions arise as limiting regimes of the dynamic problem.
\end{abstract}



\maketitle
\hypersetup{pdftitle={Dynamic resource allocation in eukaryotic Resource Balance Analysis}, pdfauthor={Saeed Sadeghi Arjmand}}
\tableofcontents

\section{Introduction}

Microbial growth constitutes a paradigmatic example of biological self-replication. According to the work \cite{Schaechter2006}, microorganisms are able to convert environmental nutrients into new cellular material at remarkable rates and with a high degree of reproducibility. The biochemical networks underlying this process have evolved under strong selective pressure, favoring cellular strategies that enhance reproductive fitness. Understanding microbial growth from a mechanistic and evolutionary perspective therefore remains a central challenge in systems biology (see, e.g.\@ \cite{Neidhardt1999}). Beyond fundamental interest, as in \cite{Brul2007, Brauner2016, Liao2016}, the ability to predict and control cellular growth is crucial for numerous applications, including antimicrobial strategies, food preservation, and biotechnological production. A fruitful perspective on cellular growth is to interpret it as a resource allocation problem (see, e.g.\@\cite{Scott2010}). Cells must distribute limited resources among competing cellular functions, such as nutrient uptake and metabolism, macromolecular synthesis, maintenance, and regulation. It is commonly assumed that, under stable environmental conditions, microorganisms have evolved allocation strategies that maximize their growth rate, thereby conferring a competitive advantage. This viewpoint has led to the development of simplified mathematical models that capture the essential trade-offs governing cellular growth while avoiding a full description of the underlying biochemical complexity.

Early resource allocation models were formulated as low-dimensional nonlinear dynamical systems, typically involving a small number of macroscopic processes and parameters derived from experimental data (see, e.g.\@ \cite{Scott2010, VanDenBerg2002, Molenaar2009, Weisse2015, Giordano2016, Erickson2017, Dourado2020}). Despite their simplicity, the models in the works~\cite{Scott2010, Molenaar2009, Giordano2016, Scott2014, Mairet2021}, successfully explained key steady-state relationships between growth rate and cellular composition, notably the dependence of ribosomal content on growth conditions. They also highlighted fundamental trade-offs between metabolic efficiency and growth yield, particularly in the context of alternative energy-producing pathways as in \cite{Molenaar2009, Basan2015, Maitra2015}. Moreover, dynamic versions of such models enabled the formulation of optimal control problems describing cellular adaptation to environmental changes, whose solutions were analyzed using Pontryagin’s Maximum Principle (PMP for short)(see, e.g.\@ \cite{Giordano2016, Carlson1991, Yegorov2018}). These studies revealed structured optimal strategies, including bang–bang controls and turnpike phenomena, and provided mechanistic interpretations of experimentally observed growth responses, such as those from \cite{Izard2015, Yegorov2019, YaboCDC2019, YaboIFAC2020, YaboMBE2020, Ewald2017}.

In parallel, constraint-based modeling approaches grounded in  resource allocation principles have emerged as powerful tools for genome-scale analysis in \cite{GoelzerCDC2009, GoelzerAutomatica2011}. These models formalize the interactions between metabolism, macromolecular synthesis, and cellular structure as systems of linear equality and inequality constraints. For a fixed growth rate, feasibility of these constraints defines a convex polyhedral set, leading to linear optimization problems that can be solved efficiently even at genome scale in \cite{BenTalNemirovski2001, Nesterov2004}. Up to our knowledge, the first genome-scale framework explicitly integrating resource allocation, known as Resource Balance Analysis (RBA for short), was developed between 2009 and 2011 in the works \cite{GoelzerCDC2009 ,GoelzerAutomatica2011}. Its predictive power was demonstrated through the construction and experimental validation of an RBA model for \emph{Bacillus subtilis} in the work~\cite{GoelzerEtAl2015}. Subsequent developments and related frameworks confirmed resource allocation as a fundamental design principle of cellular organization (see, e.g.\@\cite{MoriEtAl2016, OBrienEtAl2013, ReimersEtAl2017, GoelzerFromion2017}).

While classical RBA focuses on steady-state growth, real cells operate in dynamic environments and continuously adapt their resource allocation over time. Bridging steady-state RBA with dynamic optimal control theory therefore constitutes a natural and important extension of the framework. However, such an extension raises several theoretical challenges, including the treatment of macromolecular turnover, the preservation of convexity and tractability, and the rigorous connection between dynamic optimization and steady-state growth predictions. Building on existing RBA formulations that already incorporate macromolecular turnover, we revisit this extension from a theoretical and optimization-oriented perspective. Here, macromolecular turnover refers to the continuous degradation, renewal, and replacement of intracellular components, such as enzymes, ribosomes, and other protein complexes, which is essential for maintaining cellular homeostasis, adapting to environmental perturbations, and preserving functional integrity over time. Using the same modeling assumptions and notation as in the original framework of~\cite{GoelzerEukaryotic750182}, we reformulate metabolite and protein turnover in a way that makes explicit its impact on feasibility, convexity, and growth-rate optimization. In particular, we show that, under biologically reasonable assumptions, the inclusion of turnover preserves the linear--convex structure of the RBA feasibility problem.

Beyond steady-state descriptions, several works have emphasized the importance of dynamic resource allocation for understanding cellular adaptation and growth transitions, for instance in the series of works~\cite{YaboCDC2019, YaboMBE2020, YaboSIAM2022, YaboIFAC2020}. These approaches often rely on optimal control formulations in which cellular resources are redistributed over time in response to environmental or intracellular constraints. However, the connection between such dynamic optimization frameworks and the steady-state RBA formalism has remained largely implicit. In particular, a rigorous interpretation of steady-state RBA solutions as the outcome of an underlying dynamic optimization process is still missing, especially in the presence of macromolecular turnover and compartmentalized eukaryotic structures. Motivated by these observations, we develop in this work a unified mathematical framework that connects RBA with optimal control theory. Rather than prescribing a fixed growth rate, we interpret growth as the outcome of a time-dependent allocation of intracellular resources, governed by synthesis, degradation, and dilution mechanisms. In the dynamic formulation considered in this work, the control variable represents the fraction of the cellular translational capacity allocated over time between the synthesis of metabolic enzymes and the synthesis of other macromolecular machines. The associated optimal control problem therefore models how cells dynamically arbitrate between immediate metabolic performance and long-term cellular infrastructure. This perspective provides a dynamic foundation for resource allocation models and allows steady-state RBA solutions to be interpreted as asymptotic regimes of an underlying dynamic optimization problem.

From a theoretical perspective, we first show that the incorporation of metabolite and protein turnover preserves the linear-convex structure of the RBA feasibility problem. In particular, for a fixed growth rate, the resulting feasibility problem remains a linear program, and fundamental structural properties of RBA, including monotonicity with respect to the growth rate and the existence of a finite maximal growth rate, are preserved. We then revisit the extension of the RBA framework to eukaryotic cells, following the formulation proposed in~\cite{GoelzerEukaryotic750182}, where intracellular compartmentalization introduces additional modeling complexity. We provide a systematic derivation of compartment-level constraints and identify the assumptions required to preserve linearity and convexity. In particular, we show that, despite the presence of organelles and intracellular interfaces, the eukaryotic RBA problem remains a convex feasibility problem for fixed growth rate and therefore remains computationally tractable. Finally, our main contribution is the formulation and analysis of a dynamic optimal control problem governing cellular resource allocation. Using PMP, we derive first-order necessary conditions for optimality, characterize the structure of optimal allocation strategies, and show that the corresponding controls are generically of bang--bang type, with possible singular arcs. Moreover, we interpret steady-state RBA solutions as asymptotic envelopes of an underlying dynamic optimization process, thereby establishing a rigorous connection between dynamic resource allocation and classical growth-optimal RBA.

The paper is organized as follows. We first recall the classical RBA framework for prokaryotic cells and introduce the necessary notation. We then extend the model with metabolite and protein turnover impacts, and analyze the resulting mathematical properties. Next, we derive the RBA formulation for eukaryotic cells and discuss the impact of compartmentalization. We subsequently introduce a dynamic optimal control formulation, establish the PMP for RBA dynamics, and illustrate the theory on a minimal example highlighting the connection between dynamic allocation and steady-state growth, in particular, the optimal control formulation provides a dynamic interpretation of the RBA framework. Finally, we conclude with a broader discussion of the main theoretical and methodological contributions of this work, the principal modeling assumptions and limitations of the proposed framework, as well as several promising directions for future mathematical, computational, and biological developments.

\section{Modeling framework and notation} 
The cellular entities considered throughout this work are presented below, in which they are organized into distinct classes corresponding to metabolic, macromolecular, and structural components of the cell. Notice that in the RBA framework, the term \emph{molecular machine} refers to any macromolecular cellular component that actively catalyzes, transports, or supports a specific intracellular process, and whose synthesis requires cellular resources. Typical examples include metabolic enzymes, membrane transporters, ribosomes, chaperones, and other protein complexes involved in biosynthesis or cellular maintenance. 
\begin{enumerate}
    \item We consider a total of $N_y = N_m + N_p$ molecular machines, consisting of two distinct classes.
    The first class consists of $N_m$ metabolic enzymes and transporters involved in the metabolic network. 
    These molecular machines are denoted by $\mathbb E = (\mathbb E_1, \dotsc, \mathbb E_{N_m})$,
    with corresponding concentrations $E = (E_1, \ldots, E_{N_m})^T$,
    and associated reaction fluxes $\nu = (\nu_1, \dotsc, \nu_{N_m})^T$.
    The second class comprises $N_p$ macromolecular machines involved in non-metabolic cellular processes, 
    such as the translation apparatus. These macromolecules are denoted by $\mathbb M = (\mathbb M_1, \dotsc, \mathbb M_{N_p})$,
    with concentrations $ M = (M_1, \dots, M_{N_p})^T$.
    \item We consider a set of $N_g$ proteins $\mathbb P_G = (\mathbb P_{G_1}, \dotsc, \mathbb P_{G_{N_g}})$,
    for which the associated cellular processes are not explicitly specified. 
    The corresponding vector of concentrations is denoted by $ P_G = (P_{G_1}, \dotsc, P_{G_{N_g}})^T$.
    \item We consider a set of $N_s$ metabolites $\mathbb S = (\mathbb S_1, \dotsc, \mathbb S_{N_s})$,
    with corresponding concentrations $ S = (S_1, \ldots, S_{N_s})^T$.
    Among these metabolites, we distinguish a subset $ \mathbb B = (\mathbb B_1, \dotsc, \mathbb B_{N_b}) \subset \mathbb S$ of metabolites whose intracellular concentrations are assumed to remain fixed. The corresponding concentration vector is denoted by $ B = (B_1, \ldots, B_{N_b})^T$.
\end{enumerate}

\begin{figure}[t]
    \centering
    \includegraphics[width=\textwidth]{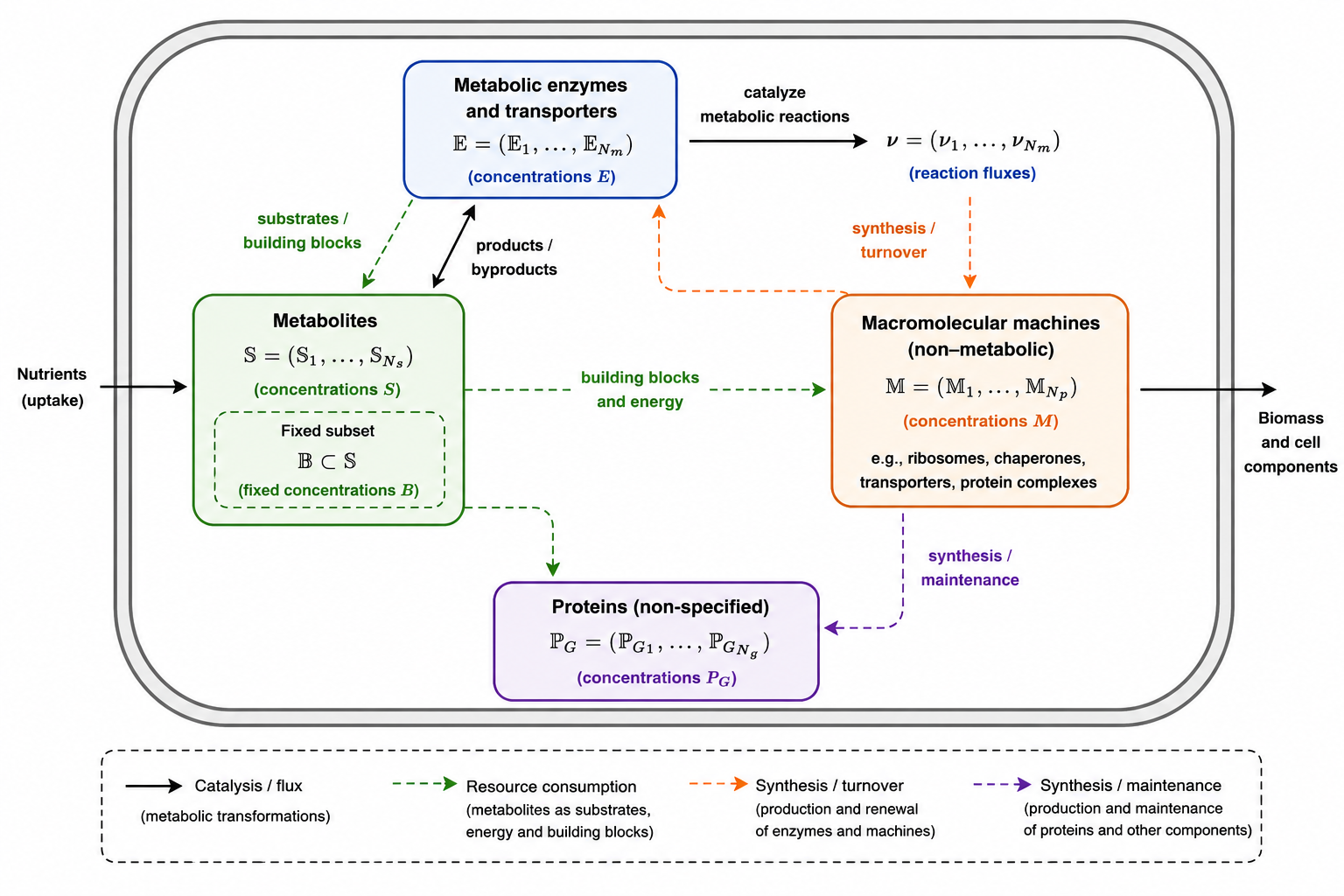}
    \caption{
    Schematic representation of the molecular entities
    and resource flows considered in the RBA framework.
    Metabolic enzymes $\mathbb E$ catalyze metabolic reactions
    with associated fluxes $\nu$ acting on metabolites $\mathbb S$.
    Macromolecular machines $\mathbb M$ support non-metabolic
    processes such as translation and protein maturation,
    thereby enabling the synthesis and maintenance
    of enzymes, structural proteins $\mathbb P_G$,
    and other cellular components.
    A subset $\mathbb B \subset \mathbb S$
    corresponds to metabolites whose concentrations
    are assumed to remain fixed.
    }
    \label{fig:cell metabolic flux}
\end{figure}

The relationships between these cellular entities are schematically illustrated in Figure~\ref{fig:cell metabolic flux}. The cardinalities introduced above are determined by the biological scope of the considered cellular model. In particular, the total number of molecular machines satisfies $N_y = N_m + N_p$, while the number of metabolites with fixed concentrations satisfies $N_b \le N_s$ since $\mathbb B \subset \mathbb S$. Apart from these structural relationships, the quantities $N_m$, $N_p$, $N_g$, $N_s$, and $N_b$ are considered independent model parameters, whose values depend on the level of biological detail and on the reconstructed cellular network. In particular, there is no requirement for the number of proteins in $\mathbb P_G$ to coincide with the number of metabolic enzymes or molecular machines.

To ensure a one-to-one correspondence between metabolic reactions and catalytic entities, we adopt the following modeling conventions. First, when a metabolic reaction can be catalyzed by two or more distinct enzymes (isoenzymes), the reaction is duplicated accordingly. For example, if a reaction $r$ can be catalyzed by enzymes $\mathbb E_1$ or $\mathbb E_2$, we introduce two reactions $r_1$ and $r_2$, catalyzed by $\mathbb E_1$ and $\mathbb E_2$, respectively. Both reactions share identical stoichiometry, but are associated with distinct catalytic capacity constraints reflecting the specific efficiencies of the corresponding enzymes. Second, when an enzymatic complex $\mathbb E_i$ is able to catalyze two or more distinct reactions, the complex is duplicated accordingly. More precisely, if $\mathbb E_i$ catalyzes reactions $r_1$ and $r_2$, we introduce two enzymatic entities $\mathbb E_{i_1}$ and $\mathbb E_{i_2}$, which catalyze $r_1$ and $r_2$, respectively. In this case, the translation process must produce both $\mathbb E_{i_1}$ and $\mathbb E_{i_2}$, thereby accounting explicitly for the resource cost associated with supporting multiple catalytic functions. These conventions allow us to associate each metabolic flux with a unique catalytic entity, leading to enzyme capacity constraints of the form $\lvert \nu_j \rvert \le k_j Y_j$ for each reaction $j$.
Importantly, this reformulation does not alter the biological interpretation of the model, but it simplifies the mathematical structure of the constraints

Under the previous modeling assumptions, the RBA feasibility problem for a prokaryotic cell, denoted by $\mathcal P_p(\mu)$, can now be formulated as follows. For a fixed vector of protein concentrations $P_G \in \mathbb{R}^{N_g}_{\ge 0},$ and a given growth rate $\mu \geq 0$, representing the amount of biomass produced per cell and per hour,
\looseness=1
\[
\mathcal P_p(\mu) \colon \left\{ 
\begin{aligned}
    &\text{ find } Y\in \mathbb R^{N_y}, \nu \in \mathbb R^{N_m},  \\
    &\text{ subject to } \\ 
    &\Romannum{1}. \quad \Omega\nu + \mu \big(C_Y^S Y+ C_B^S B + C_G^S P_G\big) = 0, \\
    &\Romannum{2}. \quad  \mu \big(C_Y^{M} Y + C_G^{M} P_G\big) \le K_{T} Y, \\
    &\Romannum{3}. \quad \abs{\nu} \le K_E Y, \\
    &\Romannum{4}. \quad C_Y^{D} Y + C_G^{D} P_G \le \bar D.
\end{aligned} \right.
\]

Each of the above constraints has a direct biological interpretation. Constraint~\textnormal{(I)} expresses the metabolic mass-balance condition: the metabolic fluxes must supply all metabolites required for biomass production and macromolecular synthesis at the prescribed growth rate $\mu$. Constraint~\textnormal{(II)} represents macromolecular synthesis capacity: the production of enzymes, molecular machines, and structural proteins cannot exceed the translational and biosynthetic capacities of the corresponding macromolecular machinery. Constraint~\textnormal{(III)} corresponds to enzyme capacity constraints, ensuring that each metabolic flux remains bounded by the catalytic efficiency and concentration of its associated enzyme. Finally, constraint~\textnormal{(IV)} imposes density limitations, reflecting physical restrictions such as cytosolic crowding, membrane occupancy, or compartment-specific volume constraints. 

All constraints appearing in the feasibility problem $\mathcal P_p(\mu)$ are understood componentwise, for which the vector of molecular machine concentrations is defined as $Y^T = \big( E^T , M^T \big)$.  The matrix $\Omega \in \mathbb{R}^{N_s \times N_m}$ denotes the stoichiometric matrix of the metabolic network, where the entry $\Omega_{ij}$ corresponds to the stoichiometric coefficient of metabolite $\mathbb S_i$ in the $j$-th enzymatic reaction. The matrix $C^S_{Y} \in \mathbb{R}^{N_s \times N_y}$ (resp. $C_S^{G} \in \mathbb{R}^{N_s \times N_g}$) collects the metabolite requirements for the synthesis of molecular machines $\mathbb Y$ (resp. proteins $\mathbb P_G$). Each coefficient $C^{S}_{Y_{ij}}$ (resp. $C^{S}_{G_{ij}}$) represents the number of molecules of metabolite $\mathbb S_i$ consumed or produced during the synthesis of one unit of $\mathbb Y_j$ (resp. $\mathbb P_{G_j}$). Accordingly, $C^{S}_{Y_{ij}}$ is positive, negative, or zero if $\mathbb S_i$ is produced, consumed, or not involved in the synthesis of $\mathbb Y_j$ (resp. $\mathbb P_{G_j}$).

Similarly, $C^S_{B} \in \mathbb{R}^{N_s \times N_b}$ denotes the matrix of metabolite requirements for the synthesis of metabolites with fixed concentrations. Each coefficient $C^{S}_{B_{ij}}$ corresponds to the number of molecules of metabolite $\mathbb S_i$ consumed or produced for the synthesis of metabolite $\mathbb B_j$. The matrices $K_T \in \mathbb{R}^{N_p \times N_p}$ and $K_E \in \mathbb{R}^{N_m \times N_m}$ are diagonal matrices. Each diagonal coefficient $K_{T_i}$ represents the efficiency of macromolecular machine $\mathbb M_i$, defined as the rate of the associated cellular process per unit concentration of $\mathbb M_i$. Similarly, $K_{E_i}$ correspond to the forward and backward catalytic efficiency of enzyme $\mathbb E_i$, respectively. The matrices $C^M_{Y} \in \mathbb{R}^{N_p \times N_y}$ and $C^M_{G} \in \mathbb{R}^{N_p \times N_g}$ describe macromolecular synthesis requirements. Each coefficient $C^{M}_{Y_{ij}}$ (resp. $C^{M}_{G_{ij}}$) typically represents the length, in amino-acid residues, of molecular machine $\mathbb Y_j$ (resp. protein $\mathbb P_{G_j}$). In specific cases, such as constraints related to protein chaperoning, these lengths may be scaled by additional factors, for instance the fraction of the proteome requiring chaperone assistance. The matrix $\bar D \in \mathbb{R}^{N_c}$ denotes the vector of density limits, where $N_c$ is the number of volume or surface constraints considered. Each component $\bar D_i$ represents the maximal allowable density of molecular entities with respect to a given volume or surface area. Densities are typically expressed as numbers of amino-acid residues per unit volume or per unit surface area.

Finally, the matrices $C^D_{Y} \in \mathbb{R}^{N_c \times N_y}$ and $C^D_{G} \in \mathbb{R}^{N_c \times N_g}$ encode density contributions of molecular entities within cellular compartments. Each coefficient $C^{D}_{Y_{ij}}$ (resp. $C^{D}_{G_{ij}}$) represents the density contribution of one unit of molecular machine $\mathbb Y_j$ (resp. protein $\mathbb P_{G_j}$) in compartment $i$. By construction, each molecular machine is assigned to a unique compartment, so that for each $j$ there exists a single index $i$ such that $C^{D}_{Y_{ij}}$ (resp. $C^{D}_{G_{ij}}$) is nonzero.

\begin{remark}
In practice, the concentration vector $B = (B_1,\dots,B_{N_b})^T$, associated with the fixed metabolites in $\mathbb B$, contains non-zero entries only for the concentrations of macromolecular components such as DNA, cell wall constituents, and lipid membranes, as well as for a limited set of metabolites whose concentrations are assumed to be fixed. Typical examples can be found in the supplementary material of~\cite{GoelzerEtAl2015}. To model reversible enzymatic reactions, two diagonal matrices of enzyme efficiencies, $K_E$ and $K_E^{0}$, are introduced to describe catalytic capacity constraints in the forward and backward directions, respectively. When an enzyme is assumed to be irreversible, the efficiency associated with the backward direction is set to zero by convention. In~\cite{GoelzerAutomatica2011, GoelzerEtAl2015}, the RBA model developed for \emph{Bacillus subtilis} incorporates two macromolecular processes in constraint~\textnormal{(II)}, namely protein translation and chaperone-assisted folding. Moreover, two density constraints are considered, corresponding to limitations on cytosolic density and membrane occupancy. The RBA framework can be further refined by incorporating additional molecular machines or by introducing new classes of constraints, such as transcriptional machinery, protein secretion systems, or other cellular processes (see, for instance~\cite{GoelzerEtAl2015}).
\end{remark}

\section{Theoretical properties of prokaryotic RBA}\label{Pre-Result-on-RBA}

We now discuss fundamental theoretical properties of the RBA optimization problem $\mathcal P_p(\mu)$, namely convexity, monotonicity with respect to the growth rate, and existence of feasible solutions. For a fixed growth rate $\mu \ge 0$, the problem $\mathcal P_p(\mu)$ is a convex feasibility problem. Indeed, all constraints defining $\mathcal P_p(\mu)$ are linear equalities or inequalities in the decision variables $(Y,\nu)$. In particular, constraint (I) is affine, constraints (II) and (IV) are linear inequalities, and constraint (III) can be rewritten as two linear inequalities. As a consequence, the feasible set \begin{equation}\label{feasible set}
\mathcal F(\mu) :=
\left\{
(Y,\nu) \in \mathbb R^{N_y}_{\ge 0} \times \mathbb R^{N_m}
\;\middle|\;
(Y,\nu) \text{ satisfies constraints (I)--(IV)}
\right\}
\end{equation}
is a closed and convex polyhedron. This property ensures that feasibility of $\mathcal P_p(\mu)$ can be tested efficiently using linear programming techniques, even at genome scale.

Let us first establish a fundamental monotonicity property of the RBA problem with respect to the growth rate.

\begin{proposition}[Monotonicity of feasibility] \label{prop:monotonicity} If the RBA problem $\mathcal P_p(\mu)$ is feasible for some $\mu \ge 0$, then $\mathcal P_p(\mu')$ is feasible for any $\mu' \in [0,\mu]$. 
\end{proposition}

\begin{proof}
Let $\mu \ge 0$ be such that $\mathcal P_p(\mu)$ is feasible, and let $(Y,\nu)$ be a feasible solution. For any $\mu' \in [0,\mu]$, define
\[
\nu' := \frac{\mu'}{\mu}\,\nu,
\]
with the convention that $\nu' = 0$ if $\mu = 0$.
We verify that $(Y,\nu')$ satisfies all constraints of $\mathcal P_p(\mu')$.
Regarding the constraint~\textnormal{(I)}, since $(Y,\nu)$ satisfies
\[
\Omega \nu + \mu \big( C_Y^{S} Y + C_B^{S} B + C_G^{S} P_G \big) = 0,
\]
we observe that by multiplying this equality by $\mu'/\mu$ yields
\[
\Omega \nu'
+ \mu' \big( C_Y^{S} Y + C_B^{S} B + C_G^{S} P_G \big) = 0,
\]
so constraint~\textnormal{(I)} holds for $\mu'$. Moreover, since all matrices, involved in constraint~\textnormal{(I)}, have nonnegative entries and $\mu' \le \mu$, we deduce
\[
\mu' \big( C_Y^{M} Y + C_G^{M} P_G \big)
\le
\mu \big( C_Y^{M} Y + C_G^{M} P_G \big)
\le K_T Y,
\]
and thus the second constraint is also satisfied. We then observe that
\[
|\nu'|
=
\frac{\mu'}{\mu} |\nu|
\le
\frac{\mu'}{\mu} K_E Y
\le K_E Y,
\]
which yields the fact that the enzyme capacity constraint (i.e.\@ constraint~\textnormal{(III)}) remains satisfied.
We finally notice that the constraint~\textnormal{(IV)} does not depend on $\mu$ and it is therefore unchanged. Thus, $(Y,\nu')$ is feasible for $\mathcal P_p(\mu')$, which concludes the proof.
\end{proof}

An immediate consequence of Proposition~\ref{prop:monotonicity} is that the set of feasible growth rates forms an interval of the form $[0,\mu_{\max}]$ or $[0,\mu_{\max})$, where $\mu_{\max}$ denotes the maximal achievable growth rate. As discussed later in Corollary~\ref{cor:mu_max} and Remark~\ref{closed growth-rate interval}, the feasible-growth set is in fact closed under the present RBA assumptions, and therefore the interval becomes $[0,\mu_{\max}]$. Under mild biological assumptions, the RBA problem admits feasible solutions for sufficiently small growth rates. In particular, for $\mu = 0$, choosing 
\[
\nu = 0,
\qquad
Y = 0
\]
satisfies the series of constraints (I)--(III). Constraint (IV) is also satisfied provided that the fixed protein concentrations $P_G$ are compatible with the density limits $\bar D$. Hence, $\mathcal P_p(0)$ is feasible whenever the basal cellular composition is physically admissible. Moreover, enzyme capacity constraints and density constraints impose upper bounds on the concentrations of molecular machines, which prevents unbounded solutions. As a result, for any fixed $\mu$, the feasible set $\mathcal F(\mu)$ is bounded. Taken together, convexity, monotonicity, and boundedness ensure that the RBA framework defines a well-posed optimization problem with a finite maximal growth rate $\mu_{\max}$. This maximal growth rate can be computed efficiently by combining linear programming with bisection or parametric optimization techniques.

\begin{corollary}[Existence of a maximal growth rate]
\label{cor:mu_max}
There exists a finite maximal growth rate
\[
\mu_{\max} := \sup \left\{ \mu \ge 0 \;\middle|\; \mathcal P_p(\mu) \text{ is feasible} \right\},
\]
such that the RBA problem $\mathcal P_p(\mu)$ is feasible for all
$\mu \in [0,\mu_{\max})$ and infeasible for all $\mu > \mu_{\max}$. Moreover, feasibility at $\mu=\mu_{\max}$ depends on whether the feasible set is closed with respect to the growth rate parameter. 
\end{corollary}

\begin{proof}
By Proposition~\ref{prop:monotonicity}, the set of feasible growth rates is an interval of the form $[0,\bar\mu]$ or $[0,\bar\mu)$ for some $\bar\mu \ge 0$, depending on whether feasibility is attained at the supremum. Feasibility at $\mu = 0$ ensures that this interval is nonempty. Furthermore, enzyme capacity constraints (III) and density constraints (IV) impose upper bounds on admissible fluxes and molecular machine concentrations. As a result, arbitrarily large values of $\mu$ cannot satisfy constraints (II) and (IV) simultaneously. Hence, the set of feasible growth rates is bounded from above, which leads us to $\mu_{\max} < +\infty$. By defining $\mu_{\max}$ as stated in the corollary, yields the desired result.
\end{proof}

\begin{remark}\label{closed growth-rate interval}
The previous corollary is stated in its most general form, allowing the feasible-growth interval to be either open or closed at its upper endpoint, depending on whether the supremum is attained. In the present RBA framework, however, for each fixed growth rate $\mu \ge 0$, the feasible set of $\mathcal P_p(\mu)$ is a closed polyhedron, and the coefficients of the constraints depend continuously on the parameter $\mu$. As a consequence, the feasible-growth set $\left\{\mu \ge 0\;\middle|\;\mathcal P_p(\mu)\text{ is feasible}\right\}$ is a closed subset of $\mathbb R_{\ge 0}$. Therefore, the supremum $\mu_{\max}$ is necessarily attained, and the interval of feasible growth rates is in fact equal to $ [0,\mu_{\max}].$ 
\end{remark}

\begin{remark}[Degeneracy and non-uniqueness of optimal solutions]
\label{rem:degeneracy}
In general, optimal solutions of the RBA problem at $\mu = \mu_{\max}$ are not unique. This non-uniqueness arises from the linear structure of the constraints and from the presence of alternative metabolic pathways, isoenzymes, or redundant molecular machines. More precisely, the feasible set $\mathcal F(\mu_{\max})$, as it is defined above in~\eqref{feasible set}, is a convex polyhedron and optimal solutions typically lie on a face of this polyhedron with dimension greater than zero. Consequently, multiple distinct flux distributions and molecular machine allocations can support the same maximal growth rate. From a biological perspective, this degeneracy reflects the existence of multiple resource allocation strategies that are equally optimal in terms of growth, allowing cells to exhibit phenotypic variability or robustness to perturbations without loss of fitness.
\end{remark}

We now show how degradation (turnover) of biological macromolecules can be integrated into the RBA framework. In principle, the approach applies to any molecular entity, including proteins, RNAs, metabolites, or macromolecular machines. For clarity, we restrict the exposition to metabolites and proteins, as other cases can be treated analogously.

Some macromolecular components of bacterial cells, such as membrane constituents or specific metabolites, are continuously damaged and must be degraded and replaced. We model this phenomenon by assuming that each metabolite $\mathbb S_j$ is degraded at a constant turnover rate $\gamma_{\mathbb S_j} > 0$. Let $S_j(t)$ denote the concentration of metabolite $\mathbb S_j$ at time $t$. Its dynamics is given by
\begin{equation*}
    \frac{\diff S_j(t)}{\diff t}
    =
    (\Omega\cdot \nu(t))_j - \bigl( \mu + \gamma_{\mathbb S_j} \bigr) S_j(t),
\end{equation*}
where $\nu_j(t)$ denotes the net metabolic production flux of $\mathbb S_j$ and $\mu$ is the cellular growth rate. At steady state, this yields
\begin{equation*}
     (\Omega\cdot \nu(t))_j = \bigl( \mu + \gamma_{\mathbb S_j} \bigr) S_j.
\end{equation*}
In particular, for metabolites $\mathbb B_j$ belonging to the subset $\mathbb B$ of compounds with fixed concentrations, we obtain
\begin{equation*}
     (\Omega\cdot \nu(t))_j = \bigl( \mu + \gamma_{\mathbb B_j} \bigr) \bar B_j,
\end{equation*}
where $\bar B_j$ denotes the prescribed steady-state concentration of $\mathbb B_j$. These additional production requirements can be directly incorporated into the metabolite balance constraints of the RBA problem. Most bacterial proteins are stable over several hours or days, with the notable exception of a limited number of proteins involved in regulatory or stress-response mechanisms. Nevertheless, even stable proteins are eventually damaged by thermal or environmental stress and are degraded by dedicated proteolytic systems.

We model now protein degradation by assigning to each molecular machine $\mathbb Y_j$ a specific turnover rate $\gamma_{\mathbb Y_j} > 0$. Let $Y_j(t)$ denote the concentration of $\mathbb Y_j$ at time $t$. Its dynamics satisfies
\begin{equation}\label{enzyme and molecular dynamics}
    \frac{\diff Y_j(t)}{\diff t}
    =
    \lambda_{\mathbb Y_j}(t) - \bigl( \mu + \gamma_{\mathbb Y_j} \bigr) Y_j(t),
\end{equation}
where $\lambda_{\mathbb Y_j}(t)$ is the synthesis flux of protein $\mathbb Y_j$ produced by the ribosome. At steady state, we obtain the relation $\lambda_{\mathbb Y_j} = \bigl( \mu + \gamma_{\mathbb Y_j} \bigr) Y_j$, which shows that protein turnover can be incorporated into the macromolecular synthesis constraints of RBA by replacing the growth rate $\mu$ with an effective rate $\mu + \gamma_{\mathbb Y_j}$ for each protein. In addition, protein degradation releases amino acids and consumes energy. More precisely, degradation of protein $\mathbb Y_j$ releases amino acids at a rate $\gamma_{\mathbb Y_j} \, \beta_{kj} Y_j$, where $\alpha_{kj}$ denotes the number of residues of amino acid $k$ in one molecule of $\mathbb Y_j$. Proteolysis also consumes ATP at a rate $\gamma_{\mathbb Y_j} \, \beta^{\prime}_{kj} Y_j$, where $\beta^{\prime}_{kj}$ denotes the number of ATP molecules required to degrade one unit of protein $\mathbb Y_j$. These contributions can be included explicitly in the metabolite balance constraints through additional source and sink terms.

\begin{remark}
For fast-growing bacteria such as \emph{Escherichia coli} or
\emph{Bacillus subtilis}, the growth rate $\mu$ is typically much larger than the turnover rates of most proteins. As a consequence, protein degradation can often be neglected without significantly affecting the predicted resource allocation obtained from RBA models.
\end{remark}

We now show that the integration of metabolite and protein turnover
does not change the convex nature of the RBA optimization problem.
As a consequence, all theoretical properties derived for the original RBA formulation, including tractability and efficient solvability, remain valid in the presence of degradation. Recall that the feasible set of the prokaryotic RBA problem $\mathcal P_p(\mu)$ is defined by a system of linear equalities and inequalities in the decision variables $(Y,\nu)$. Convexity of the feasible region follows directly from linearity of the constraints. At steady state, metabolite turnover modifies the mass-balance equations by replacing the standard growth dilution term $\mu S_j$ with an effective dilution term $(\mu + \gamma_{\mathbb S_j}) S_j$. Equivalently, the metabolite balance constraint
\[
\Omega \nu + \mu \bigl( C_Y^S Y + C_B^S B + C_G^S P_G \bigr) = 0
\]
is replaced by
\[
\Omega \nu
+ \mu \bigl( C_Y^S Y + C_B^S \bar B + C_G^S P_G \bigr)
+ \Gamma^S_Y Y + \Gamma^S_B \bar B + \Gamma^S_{P_G} P_G
= 0,
\]
where $\Gamma_Y^S \in \mathbb R^{N_s \times N_y}$, $\Gamma_B^S \in \mathbb R^{N_s \times N_b}$, $\Gamma_{P_G}^S \in \mathbb R^{N_s \times N_g}$, are diagonal (or block-diagonal) matrices encoding metabolite requirements associated with turnover of molecular machines, fixed metabolites, and unspecific proteins, respectively. Each coefficient of these matrices corresponds to the amount of metabolite released or consumed per unit time due to degradation processes. Similarly, protein turnover modifies the macromolecular synthesis constraints by replacing the growth rate $\mu$ with an effective rate $\mu + \gamma_{\mathbb Y_j}$ for each molecular machine $\mathbb Y_j$. At steady state, this yields the linear relation
\[
\lambda_{\mathbb Y_j} = (\mu + \gamma_{\mathbb Y_j}) Y_j.
\]

The above correction is essential to ensure dimensional consistency and a proper interpretation of protein turnover within the RBA framework. Indeed, turnover acts at the level of individual molecular machines and proteins, and not at the level of macromolecular processes. As a consequence, degradation-induced synthesis requirements must be incorporated inside the linear maps describing macromolecular composition. More precisely, for each molecular machine $\mathbb Y_j$ with turnover rate $\gamma_{\mathbb Y_j}$, the additional synthesis demand induced by degradation is proportional to its macromolecular composition. This contribution is encoded in the matrix $\Gamma_Y^M \in \mathbb R^{N_p \times N_y}$, whose coefficients are defined by $(\Gamma_Y^M)_{ij} = \gamma_{\mathbb Y_j} \, C^M_{Y_{ij}}$, where $C^M_{Y_{ij}}$ denotes the contribution of molecular machine $\mathbb Y_j$ to macromolecular process $i$. An analogous construction yields the matrix $\Gamma_{P_G}^M \in \mathbb R^{N_p \times N_g}$ for the turnover of proteins belonging to the pool $\mathbb P_G$. With these definitions, the macromolecular synthesis constraint with protein turnover thus satisfies
\[
\mu \bigl( C_Y^M Y + C_G^M P_G \bigr)
+ \Gamma_Y^M Y
+ \Gamma_{P_G}^M P_G
\le K_T Y.
\]
This formulation makes explicit that protein turnover introduces additional linear synthesis demands without coupling distinct macromolecular processes. In particular, no nonlinear terms are generated, and the constraint remains affine in the decision variables. Hence, the inclusion of protein degradation preserves the convexity of the feasible set of the RBA problem. Additional source and sink terms associated with protein degradation, such as amino-acid release and ATP consumption, enter the metabolite balance constraints as linear contributions proportional to $Y$. Therefore, these terms do not introduce any nonlinear coupling between decision variables.

We thus observe that all turnover-related modifications affect the RBA formulation only through linear transformations of existing constraints (i.e.\@ no bilinear or nonlinear terms are introduced). Hence, the feasible set of the RBA problem with macromolecular turnover remains a convex polyhedron. In particular, for a fixed growth rate $\mu$, the RBA feasibility problem with degradation can still be formulated as a linear programming problem.

\begin{proposition}[Convexity preservation of RBA under macromolecular turnover]
\label{prop:convexity_turnover}
Let $\mu \ge 0$ be a fixed growth rate and $P_G \in \mathbb R^{N_g}_{\ge 0}$ a fixed vector of protein concentrations. Consider the RBA problem extended to include metabolite and protein turnover, defined as follows,
\[
\mathcal P_p^{\mathrm{turn}}(\mu) \colon
\left\{
\begin{aligned}
    &\mathrm{ find }\,\, (Y,\nu) \in \mathbb R^{N_y}_{\ge 0} \times \mathbb R^{N_m}, \\[0.3em]
    &\mathrm{subject}\,\mathrm{to} \\[0.3em]
    &\Romannum{1}. \quad
    \Omega \nu
    + \mu \bigl( C_Y^S Y + C_B^S \bar B + C_G^S P_G \bigr)
    + \Gamma_Y^S Y + \Gamma_B^S \bar B + \Gamma_{P_G}^S P_G
    = 0, \\[0.6em]
    &\Romannum{2}. \quad
    \mu \bigl( C_Y^M Y + C_G^M P_G \bigr)
    + \Gamma_Y^M Y
    + \Gamma_{P_G}^M P_G
    \le K_T Y, \\[0.6em]
    &\Romannum{3}. \quad
    |\nu| \le K_E Y, \\[0.6em]
    &\Romannum{4}. \quad
    C_Y^D Y + C_G^D P_G \le \bar D .
\end{aligned}
\right.
\]
Then, for fixed $\mu$, the feasible set of $\mathcal P_p^{\mathrm{turn}}(\mu)$ is a convex polyhedron. In particular, the feasibility problem associated with $\mathcal P_p^{\mathrm{turn}}(\mu)$ can be formulated as a linear program.
\end{proposition}

\begin{proof}
We show that all constraints defining the feasible set of
$\mathcal P_p^{\mathrm{turn}}(\mu)$ are linear equalities or inequalities in the decision variables $(Y,\nu)$. Constraint (\Romannum{1}) is an affine equality constraint. Indeed, $\Omega \nu$ is linear in $\nu$, while all remaining terms are linear functions of $Y$, since $\bar B$ and $P_G$ are fixed parameters and the turnover matrices $\Gamma_Y^S$, $\Gamma_B^S$, and $\Gamma_{P_G}^S$ are constant. Constraint (\Romannum{2}) is an affine inequality. The macromolecular synthesis demand consists of a growth-related term scaled by $\mu$ and additional turnover-induced terms, all of which are linear in $Y$. The right-hand side $K_T Y$ is linear since $K_T$ is diagonal. Constraint (\Romannum{3}) is equivalent to the pair of linear inequalities $-\,K_E Y \le \nu \le K_E Y$, and is therefore affine, and finally Constraint (\Romannum{4}) is linear in $Y$. Since the feasible set is defined by the intersection of finitely many affine subspaces, it is a convex polyhedron. Consequently, for fixed $\mu$, the feasibility problem associated with $\mathcal P_p^{\mathrm{turn}}(\mu)$ is a linear programming problem.
\end{proof}

\begin{proposition}[Existence of a maximal growth rate]\label{prop:mu_max}
There exists a finite maximal growth rate
\[
\mu_{\max} \in \mathbb R_{\ge 0}
\]
such that the RBA problem with macromolecular turnover
\(\mathcal P_p^{\mathrm{turn}}(\mu)\) is feasible for all
\(0 \le \mu \le \mu_{\max}\), and infeasible for all \(\mu > \mu_{\max}\).
\end{proposition}

\begin{proof}
For any fixed growth rate \(\mu \ge 0\), the feasible set of
\(\mathcal P_p^{\mathrm{turn}}(\mu)\) is defined by a finite system of linear equalities and inequalities in the decision variables \((Y,\nu)\). Hence, for fixed \(\mu\), feasibility of \(\mathcal P_p^{\mathrm{turn}}(\mu)\) is a linear programming feasibility problem. We first note that \(\mathcal P_p^{\mathrm{turn}}(0)\) is feasible. Indeed, setting \(Y = 0\) and \(\nu = 0\) satisfies all constraints, since all turnover-related terms are linear and nonnegative, and all density and capacity constraints are satisfied trivially. Next, define the feasible-growth set
\[
\mathcal M := \{ \mu \ge 0 \mid \mathcal P_p^{\mathrm{turn}}(\mu)
\text{ is feasible} \}.
\]
By the same monotonicity argument as in Proposition~\ref{prop:monotonicity}, which remains valid in the presence of turnover since the turnover terms enter linearly and preserve the affine dependence on $\mu$, the feasible-growth set $\mathcal M$ is an interval of the form $[0,\mu_{\max}]$ or $[0,\mu_{\max})$.

It remains to show that the feasible-growth set $\mathcal M$ is bounded from above and closed in $\mathbb R_{\ge 0}$. To do so, consider the macromolecular synthesis constraint 
\[
\mu \bigl( C_Y^M Y + C_G^M P_G \bigr)
+ \Gamma_Y^M Y
+ \Gamma_{P_G}^M P_G
\le K_T Y,
\]
for which all matrices involved have nonnegative coefficients, and the diagonal matrix \(K_T\) has strictly positive diagonal entries. For any feasible solution with \(Y \neq 0\), this inequality implies an upper bound on \(\mu\) that depends only on the efficiency coefficients contained in \(K_T\) and the macromolecular composition matrices. Therefore, feasibility cannot hold for arbitrarily large values of \(\mu\). Moreover, by the same closedness argument as in Remark~\ref{closed growth-rate interval}, for each fixed $\mu$, the feasible set of $\mathcal P_p^{\mathrm{turn}}(\mu)$ is a closed polyhedron, and all coefficients depend continuously on $\mu$. Therefore, the feasible-growth set $\mathcal M$ is closed in $\mathbb R_{\ge 0}$. Since $\mathcal M$ is nonempty and bounded from above, it follows that $\mu_{\max} := \sup \mathcal M < +\infty$. Hence, we deduce $\mathcal M = [0,\mu_{\max}].$
\end{proof}

\section{RBA for eukaryotic cells}

In this section, we summarizes the extension of the RBA
framework to eukaryotic cells.
The presentation follows the developments introduced in
\cite[Sections~3.1--3.6]{GoelzerEukaryotic750182}, but despite the summarization, it is intentionally kept at a high level.
Our objective is to expose the modeling principles, the structure of the constraints, and the resulting optimization problem, while referring the reader to the work~\cite{GoelzerEukaryotic750182} for detailed derivations and biological justifications.

Eukaryotic cells are substantially more complex than prokaryotic cells. The precise regulation of a large diversity of cellular processes remains only partially understood, even for well-established mechanisms such as transcription. From the perspective of RBA, the main challenge lies in systematically integrating the presence of organelles into the framework. This requires an explicit treatment of the localization of metabolism, macromolecular processes, molecular machines, and the associated constraints across multiple cellular compartments. In most genome-scale metabolic models, organelles are typically taken into account only through the localization of metabolites. In some recent reconstructions, such as the latest release of the yeast consensus model, reaction localization is added as an annotation in the reaction names. However, the biomass reaction usually represents the composition of the entire cell as a single entity, without distinguishing between the composition of individual organelles. As a consequence, the allocation of cellular resources to organelles is implicitly assumed to be fixed and independent of environmental conditions.

In contrast, the RBA formulation developed here removes this assumption. The cellular investment into organelles is explicitly included in the resource allocation problem, allowing organelle composition and size to adapt to growth conditions in a principled manner. Eukaryotic cells exhibit a much higher degree of structural and functional organization than prokaryotic cells. In particular, most cellular processes are spatially organized within membrane-bound organelles, such as mitochondria, nuclei, or endoplasmic reticulum. As a consequence, metabolism, macromolecular synthesis, and transport processes take place in distinct compartments, each characterized by its own physicochemical constraints. From the perspective of RBA, this raises a fundamental modeling challenge such as, for instance, how to integrate compartmentalization and organelle structure into a resource allocation framework, without losing the mathematical properties (linearity, convexity, tractability) that make RBA suitable for genome-scale applications. Classical genome-scale metabolic models typically account for compartments by labeling metabolites or reactions, while the biomass composition remains defined at the whole-cell level. As a result, the relative investment of cellular resources into different organelles is implicitly fixed. In contrast, the eukaryotic RBA framework explicitly treats organelle investment as a variable determined by resource allocation constraints.

\subsection{Modeling organelles}

Let us now discuss the general schema of the eukaryotic cell first steps by more details. To do so, we define a eukaryotic cell as a complex system whose cytoplasm contains multiple organelles. These organelles are represented through the introduction of $N_{\mathrm{com}}$ compartments, with $N_{\mathrm{com}} \ge 2$, since we always assume the presence of at least mitochondria, as illustrated in Figure~\ref{fig:eukaryotic_compartments}. The cell is partitioned into $N_{\mathrm{com}}$ compartments with volumes $V^z$, for $z \in \{1,\dotsc,N_{\mathrm{com}}\}$, and into $N^{\mathrm{int}}$ interfaces with surface areas $\mathcal A^z$, indexed by pairs $z \in \{0 \leftrightarrow i, \dotsc, j \leftrightarrow k\}$. In addition, the cell contains a set of metabolites $\mathbb S$ and a set of molecular machines $\mathbb Y$, which may be localized in different compartments (e.g.\@ the subsets $\mathbb S^k$ and $\mathbb Y^k$ associated with the $k$-th compartment).

In the illustration (see, Figure~\ref{fig:eukaryotic_compartments}), two organelles are represented. The first organelle, of volume $V^i$, is bounded by a single membrane with surface area $\mathcal A^{0 \leftrightarrow i}$. The second organelle is delimited by two membranes (for instance, an outer and an inner membrane) with surface areas $\mathcal A^{0 \leftrightarrow j}$ and $\mathcal A^{j \leftrightarrow k}$, respectively. These membranes define an intermembrane space of volume $V^j$ and a matrix of volume $V^k$. Consequently, the second organelle consists of two compartments, with volumes $V^j$ and $V^k$, associated with the interfaces $\mathcal A^{0 \leftrightarrow j}$ and $\mathcal A^{j \leftrightarrow k}$, respectively. In addition to compartments, it is necessary to introduce $N_{\mathrm{int}}$ interfaces describing the boundaries between compartments. At time $t$, the $i$-th compartment is characterized by a volume $V_i(t)$. Each compartment is connected to other compartments through interfaces, each interface being associated with a surface area denoted by $\mathcal A^{\ast\leftrightarrow i}$. Let $V_c(t)$ denote the total cell volume at time $t$. It is defined as
\begin{equation}
    V_c(t)
    =
    V_0(t)
    +
    \sum_{i=1}^{N_{\mathrm{com}}} V_i(t),
\end{equation}
where $V_0(t)$ corresponds to the volume of the cell not contained in any of the $N_{\mathrm{com}}$ compartments. This volume is classically referred to as the cytosol, and notice that the concentrations of molecular entities present in the cytosol are defined with respect to the cytosolic volume $V_0(t)$.

\begin{figure}[ht]
\centering
\begin{tikzpicture}[scale=0.9, every node/.style={font=\small}]

\tikzset{
    cell/.style={draw=black!55, thick},
    compartment/.style={draw=black!55, thick},
    interface/.style={draw=black!55, dashed, thick},
    arrow/.style={->, thick, black!55},
    Y0/.style={circle, fill=red!60, inner sep=2.5pt},
    Yi/.style={circle, fill=purple!60, inner sep=2.5pt},
    Yk/.style={circle, fill=red!60, inner sep=2.5pt},
    Yg/.style={circle, fill=teal!60, inner sep=2.5pt},
}

\draw[cell] (0,0) ellipse (8cm and 4.5cm);
\node at (3.8,2.8) {$V^c = \displaystyle\sum_{j=0}^{N_{\mathrm{com}}} V^j$};

\node at (5.5,0.5) {$V^0$};

\draw[compartment] (-3,1.2) ellipse (2.6cm and 1.6cm);
\node at (-4.5,1.7) {$V^i$};

\draw[interface,->] (-4.9,-0.5) -- (-4.6,-0.1);
\node at (-5.2,-0.8) {$\mathcal A^{0\leftrightarrow i}$};

\node at (-3.1,0.6) {$\mathbb S^i$};

\node at (-2.2,1.2) {$\mathbb Y^i$};

\foreach \x/\y in {-3.6/1.6, -3.2/1.8, -3.0/1.4} {
    \node[Yi] at (\x,\y) {};
}
\node[Yi] at (-3.9,1.3) {};
\node[Yg] at (-2.2,2.1) {};
\node[Yg] at (-2,1.9) {};
\node[Yg] at (-2.3,1.8) {};
\node at (-4.1,0.8) {$\mathbb Y^i_3$};
\node at (-2.7,1.9) {$\mathbb Y^i_1$};


\draw[compartment] (2.5,-1.6) ellipse (3.0cm and 1.6cm);
\node at (4.9,-0.99) {$V^j$};

\draw[compartment] (2.5,-1.6) ellipse (2.3cm and 1.1cm);
\node at (4.2,-1.9) {$V^k$};

\draw[interface,->] (0.2,-3.3) -- (0.5,-2.9);
\node at (0,-3.5) {$\mathcal A^{0\leftrightarrow j}$};

\draw[interface,->] (2.3,-3.9) -- (2.5,-2.9);
\node at (2.3,-3.9) {$\mathcal A^{j\leftrightarrow k}$};

\node at (1.2,-2.0) {$\mathbb S^k$};
\node at (-2,-1.2) {$\mathbb S^0$};

\node at (2.1,-1.3) {$\mathbb Y^k$};

\foreach \x/\y in {2.2/-1.9, 2.6/-2.1, 2.9/-1.8, 2.4/-1.4} {
    \node[Yk] at (\x,\y) {};
}
\node[Yg] at (3.2,-1.3) {};
\node at (3.6,-1.2) {$\mathbb Y^k_1$};
\node at (3.0,-2.3) {$\mathbb Y^k_2$};

\node at (0.4,0.1) {$\mathbb Y^0$};

\foreach \x/\y in {0.4/1.1, 0.7/0.7} {
    \node[Yg] at (\x,\y) {};
}
\foreach \x/\y in {1.2/0.8, 1.5/0.6, 1.7/1.0, 1.4/1.2} {
    \node[Y0] at (\x,\y) {};
}

\node at (0,1.3) {$\mathbb Y^0_1$};
\node at (1.8,1.4) {$\mathbb Y^0_2$};

\draw[<->, thick, black!55] (-2.1,1.05) -- (0.1,0.2);
\draw[<->, thick, black!55] (0.4,-0.1) -- (1.8,-1.3);

\draw[<->, thick, black!55] (-3.1,0.4) -- (-2.2,-1);
\draw[<->, thick, black!55] (-1.9,-1.4) -- (0.9,-2);

\end{tikzpicture}
\caption{The eukaryotic cell with interaction of two compartments.}
\label{fig:eukaryotic_compartments}
\end{figure}

To model the cell, we introduce a catalogue of molecular entities
(defined in detail below), and assume that each entity is associated
either with a specific compartment or with an interface between two
compartments. The abundance of each entity is assumed to be known in every compartment and interface in which it is localized. Let $n_{X^i}$ denote the number of entities of type $\mathbb X$ in the $i$-th compartment. The corresponding concentration in this compartment is defined as 
\[
[X^i]^i = \frac{n_{X^i}}{V^i},
\]
where $V^i$ denotes the volume of compartment $i$. We further define the concentration of $\mathbb X$ in compartment $i$ with respect to the total cell volume $V^c$ as
\[
[X^i]^{c} = \frac{n_{X^i}}{V^c}.
\]
The total concentration of entity $\mathbb X$ with respect to the cell volume is then given by
\[
[X]^c = \sum_{i=0}^{N_{\mathrm{com}}} \frac{n_{X^i}}{V^c},
\]
where $n_{X^0}$ corresponds to the number of entities $\mathbb X$ located in the cytosol. Finally, for an interface between compartments $i$ and $j$, the number of entities $\mathbb X$ associated with this interface is denoted by $n_{X^{i \leftrightarrow j}}$.

Concerning the compartments, interfaces, and geometric variables, we notice that the cell is modeled as a collection of $N_{\mathrm{com}}$ compartments, including the cytosol and organelles. Each compartment $i$ is characterized by a volume $V^i(t)$, and pairs of compartments are separated by interfaces (typically membranes) with surface areas $\mathcal A^{i\leftrightarrow j}(t)$. In order to formulate steady-state constraints compatible with exponential growth, all concentrations are expressed with respect to the total cell volume $V^c(t)$. This normalization allows compartment volumes and interface surfaces to vary with growth conditions while preserving linear relations between variables. Physical constraints on intracellular crowding and membrane occupancy are captured through density constraints. These constraints limit the amount of macromolecular material that can be contained in a compartment or associated with an interface. Under the assumption that these density constraints are saturated at steady state, compartment volumes and surface areas can be expressed as linear functions of macromolecular concentrations. This assumption plays a crucial role in preserving the linear structure of the RBA problem.

Metabolic reactions are localized within compartments. Each compartment has its own internal metabolic network, as well as exchange fluxes with neighboring compartments through interfaces. At steady state, metabolite balance equations are written independently for each compartment. Exchange fluxes appear with opposite signs in the mass balance equations of the two adjacent compartments, ensuring global mass conservation at the cell level. Despite the apparent increase in the number of constraints, all metabolite balances remain linear in fluxes and enzyme concentrations. Enzymes and macromolecular machines are assumed to operate locally in the compartments where they are localized. As in prokaryotic RBA, catalytic capacity constraints relate metabolic fluxes to enzyme concentrations via efficiency coefficients. Macromolecular processes, such as translation, protein import, or protein folding, are described by linear constraints linking process fluxes to the availability of the corresponding molecular machines. Targeting of proteins to specific compartments is explicitly accounted for through additional macromolecular activities consuming cellular resources.

Importantly, compartmentalization does not introduce nonlinear couplings between processes. All constraints remain affine in the decision variables. For each compartment, the RBA framework yields a set of constraints comprising, such as metabolite mass balances, enzyme capacity constraints, macromolecular synthesis constraints, and density constraints linking molecular content to compartment geometry. After normalization by the cell volume and elimination of auxiliary geometric variables, these compartment-level constraints can be assembled into a single cell-scale formulation. The resulting system closely mirrors the structure of the prokaryotic RBA problem, with additional linear constraints encoding compartmental organization. The key conceptual result is that compartmentalization does not alter the theoretical nature of the RBA problem. Under biologically reasonable assumptions, the eukaryotic RBA framework remains a linear, convex feasibility problem for a fixed growth rate. This ensures that genome-scale eukaryotic RBA models can, in principle, be solved efficiently using linear programming techniques.

Regarding the formulation of the eukaryotic RBA model, we first notice that a mitochondrion cannot be represented by a single compartment. Instead, at least two compartments must be introduced, such as the intermembrane space and the matrix. The matrix is associated with a single interface corresponding to the inner mitochondrial membrane. The intermembrane space has two interfaces: one with the matrix (the inner membrane) and one with the cytosol corresponding to the outer mitochondrial membrane.

\subsection{Formulation of the RBA problem for eukaryotic cells}

Introducing compartments and interfaces leads to a formulation of RBA that remains formally close to the prokaryotic case, provided that all quantities are normalized with respect to the total cell volume. We present in below the RBA feasibility problem for eukaryotic cells, by considering the following arguments.

\begin{itemize}
    \item We consider a set of $N_{\mathrm{com}}$ cellular compartments indexed by
    $\mathcal I_V = \{0,1,\dotsc,N_{\mathrm{com}}\}$,
    together with a set of $N^{\mathrm{int}}$ interfaces indexed by
    $\mathcal I_A = \{\,0 \leftrightarrow i,\; \dotsc,\; j \leftrightarrow k\,\}$.
    The index set $\mathcal I_A$ contains all interfaces separating pairs of compartments.
    We denote by
    \[
    N_c = N_{\mathrm{com}} + 1 + N^{\mathrm{int}}
    \]
    the total number of compartments and interfaces considered in the model.
    
    \item Let $Y \in \mathbb{R}^{N_y}_{\ge 0}$ denote the vector collecting the concentrations of all molecular machines across all cellular compartments and interfaces. Here, the quantities $N_m$ and $N_p$ refer to the numbers of distinct classes of metabolic enzymes and non-metabolic molecular machines introduced previously, independently of compartmental localization.
    In the eukaryotic setting, however, a given molecular machine may be associated with one or several compartment-specific copies. Consequently, $N_y$ denotes the total number of compartment-specific molecular machine instances, and therefore satisfies $N_y \ge N_m + N_p$,
    with equality when each molecular machine is assigned to a unique cellular compartment or interface. 

    \item $P_G \in \mathbb{R}^{N_g}_{\ge 0}$ denotes the vector of concentrations of the proteins in $\mathbb P_G$,
    defined over all compartments and interfaces,
    and $f_V$ denotes the vector of normalized compartment volumes with respect to the total cell volume,
    $f_V = \bigl( f_V^i \bigr)_{i \in \mathcal I_V}^T.$
    Similarly, $f_A$ denotes the vector of normalized interface surface areas with respect to the total
    cell volume,
    $f_A = \bigl( f_A^i \bigr)_{i \in \mathcal I_A}^{T}.$
    Finally, we define
    \[
    f = \begin{pmatrix} f_V \\ f_A \end{pmatrix}
    \]
    as the concatenation of the normalized volume and surface fractions.
\end{itemize}
Having established the modeling framework for compartmentalized cells, including metabolite balances, macromolecular synthesis, and density constraints, we now summarize these elements into a single feasibility problem. The objective is not to optimize directly, but to characterize whether a given growth rate can be supported by the available cellular resources. For a fixed growth rate $\mu \ge 0$, the eukaryotic RBA feasibility problem $\mathcal P^{e}_{\mathrm{rba}}(\mu)$ is defined as
\[
\mathcal P^{e}_{\mathrm{rba}}(\mu) \colon
\left\{ 
\begin{aligned}
    &\text{ find } Y\in \mathbb R^{N_y}, \nu \in \mathbb R^{N_m}, f \in \mathbb R^{N_c}  \\
    &\text{ subject to } \\ 
    &\Romannum{1}. \quad \Omega\nu + \mu \big(C_Y^S Y+ C_B^S B + C_G^S P_G + C_F^S f \hat B\big) = 0, \\
    &\Romannum{2}. \quad  \mu \big(C_Y^{M} Y + C_G^{M} P_G\big) \le K_{T} Y, \\
    &\Romannum{3}. \quad \abs{\nu} \le K_E Y, \\
    &\Romannum{4}. \quad C_Y^{D,iq} Y + C_G^{D,iq} P_G - C_F^{D,iq} f \le 0, \\
    &\Romannum{5}. \quad C_Y^{D,eq} Y + C_G^{D,eq} P_G - C_F^{D,eq} f = 0,  \\
    &\Romannum{6}. \quad C_F^{F} f - \bar C = 0,  \\
    &\Romannum{7}. \quad \underline f_{V} \le I_{V} f \le \bar f_{V},
\end{aligned} 
\right.
\]
where constraints (\Romannum{1}) and (\Romannum{2}) generalize the classical RBA metabolite balance
and enzyme capacity constraints to the case of compartmentalized cells. Constraints (\Romannum{4}) and (\Romannum{5}) encode density limitations associated with compartments and interfaces. Some of these constraints may be non-saturated (e.g.\ cytosolic crowding), while others are assumed to be always saturated (e.g.\ membrane occupancy). Constraint (\Romannum{6}) enforces the conservation of total cell volume and surface through normalization of compartment fractions. Finally, constraint (\Romannum{7}) imposes lower and upper bounds on the relative sizes of compartments and interfaces. The vectors $\underline f_{V}$ and $\bar f_{V}$ denote the minimal and maximal normalized volumes of compartments respectively, and the matrix $I_{V}$ is defined such that $f_V =  I_V f$.

For a fixed growth rate $\mu$, the feasible set of $\mathcal P^{e}_{\mathrm{rba}}(\mu)$ is again a convex polyhedron. Consequently, the eukaryotic RBA feasibility problem remains a linear feasibility problem, similarly to the prokaryotic case. The proof follows from the same arguments used in the prokaryotic framework, since all constraints remain affine or linear with respect to the decision variables  (see also~\cite{GoelzerEukaryotic750182} for further details on the eukaryotic RBA formulation). For brevity, we do not reproduce these arguments here.

\section{Dynamic resource allocation framework}

While classical RBA focuses on steady-state growth
and static feasibility problems, microbial cells operate in dynamically changing environments and must continuously adapt their allocation of biosynthetic resources. In this section, we therefore extend the RBA framework to a more precise dynamic setting and formulate an optimal control problem that captures the temporal regulation of resource allocation at the cellular level. We adopt the widely used hypothesis that microbial populations have evolved resource allocation strategies that maximize biomass production over time \cite{EdwardsIbarraPalsson2001, LewisEtAl2010}. Within this perspective, cellular growth is interpreted as the outcome of an optimization process driven by evolutionary pressure. Mathematically, this hypothesis leads naturally to an optimal control formulation, in which the objective is to maximize the final cell volume $V(T)$ over a fixed time horizon $[0,T]$. We consider the cellular dynamics given by~\eqref{enzyme and molecular dynamics}, which describe the temporal evolution of molecular machines under growth dilution and macromolecular turnover.

To formulate the dynamic resource allocation problem, we first introduce the state dynamics and control variables. More precisely, we rewrite~\eqref{enzyme and molecular dynamics} in a form suitable for optimal control analysis. For each molecular machine $\mathbb Y_j$, the concentration dynamics read
\begin{equation}\label{control dynamics}
    \frac{\diff Y_j(t)}{\diff t}
    =
    \lambda_{\mathbb Y_j}(t)
    - \bigl( \mu + \gamma_{\mathbb Y_j} \bigr) Y_j(t)
    =
    \begin{cases}
        \alpha(t)\,\nu_{\mathbb E_j}(t)
        - \bigl( \mu + \gamma_{\mathbb Y_j} \bigr) E_j(t), \\[0.4em]
        \bigl(1-\alpha(t)\bigr)\,\nu_{\mathbb M_j}(t)
        - \bigl( \mu + \gamma_{\mathbb Y_j} \bigr) M_j(t),
    \end{cases}
\end{equation}
where $\nu_{\mathbb E_j}(t)$ and $\nu_{\mathbb M_j}(t)$ denote the effective synthesis fluxes of enzymes and macromolecular machines, respectively, driven by ribosomal activity. The corresponding resource allocation mechanism is illustrated schematically in Fig.~\ref{fig:dynamic_allocation}. The scalar function $\alpha(t) \in [0,1]$ (hereafter referred to as the control) represents the fraction of translational capacity allocated to enzyme synthesis at time $t$, while the complementary fraction $1-\alpha(t)$ is allocated to the synthesis of other macromolecular machines. Moreover, the control function $\alpha(\cdot)$ is assumed to take values in a compact and convex admissible set $\mathcal U$, reflecting physical and biological limitations on resource allocation. Although the formulation involves a single scalar control for clarity, the approach readily extends to multiple allocation variables corresponding to distinct translation or synthesis processes. From a biological point of view, $\alpha(t)$ encodes the global strategy by which the cell distributes its limited biosynthetic resources. From a control-theoretic perspective, $\alpha$ is treated as the control input governing the dynamics of the system.  Although a single scalar control is considered here for simplicity, this formulation already captures the fundamental trade-offs between metabolic capacity and cellular infrastructure. Moreover, it naturally extends to more refined models involving multiple control variables associated with different protein classes.

\begin{figure}[t]
    \centering
    \includegraphics[width=\textwidth]{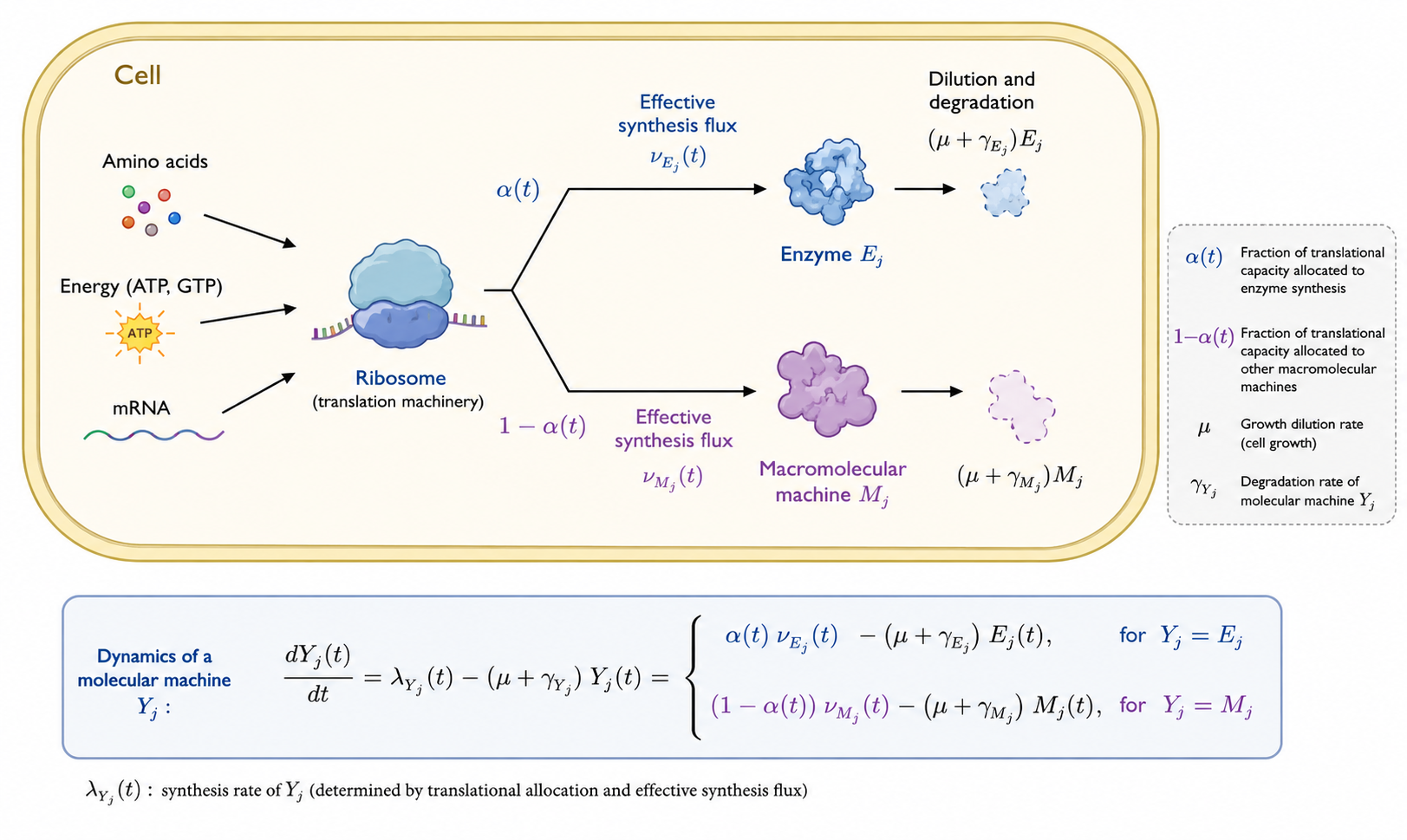}
    \caption{
    Schematic representation of the dynamic resource allocation mechanism.
    At each time $t$, a fraction $\alpha(t)$ of the translational capacity
    is allocated to enzyme synthesis, while the complementary fraction
    $1-\alpha(t)$ is allocated to the synthesis of other macromolecular
    machines. Both classes are subject to growth dilution and turnover.
    }
    \label{fig:dynamic_allocation}
\end{figure}

Having introduced the dynamic state variables and the control input, we can now formulate the corresponding optimal control problem. Notice first that the cell volume evolves according to $\dot V(t) = \mu(t)\, V(t)$. 
Since the initial volume $V(0)$ is fixed, maximizing the final volume $V(T)$ is equivalent to maximizing the accumulated growth rate over the interval $[0,T]$. We therefore introduce the cost functional
\[
\mathcal J(\alpha)
=
\int_{0}^{T} \mu\bigl(E(t), M(t)\bigr)\, \diff t,
\]
where the instantaneous growth rate $\mu$ depends on the current cellular state. In practice, $\mu(t)$ may be interpreted as the maximal growth rate compatible with the cellular resource constraints at time $t$, as determined by an underlying RBA feasibility problem. The dynamic resource allocation problem is formulated as an optimal control problem over a finite time horizon $[0,T]$. The state variables correspond to the concentrations of enzymes and metabolic machineries, whose dynamics are governed by synthesis and degradation processes, and the control variable represents the allocation of translational capacity among these cellular components. The objective is to maximize the cumulative growth over the considered time interval, quantified by the integral of the instantaneous growth rate. \begin{equation}\label{OCP}
\left\{
\begin{aligned}
    & \text{maximize} \quad
    \mathcal J(\alpha)
    = \int_{0}^{T} \mu\bigl(E(t), M(t)\bigr)\, \diff t, \\[0.3em]
    & \text{subject to} \quad
    \text{the dynamics~\eqref{control dynamics}}, \\[0.3em]
    & E_j(0) \ge 0, \quad M_j(0) \ge 0, \\[0.3em]
    & \alpha \in \mathcal U,
\end{aligned}
\right.
\end{equation}
where $\mathcal U$ denotes the set of admissible controls,
defined as the set of Lebesgue-measurable functions $\alpha \colon [0,T] \to [0,1]$, and notice that the instantaneous growth rate  $\mu\bigl(E(t), M(t)\bigr)$ should be understood as a smooth proxy for the growth potential induced by the current macromolecular composition of the cell. In particular, it may be interpreted as the maximal growth rate compatible with the underlying RBA constraints for fixed values of $E$ and $M$, without requiring an explicit solution of the associated feasibility problem at each time. The problem does not involve terminal constraints or explicit path constraints.  Nevertheless, path constraints could be incorporated to enforce additional physical or biological limitations, such as bounds on total protein content or ribosomal capacity. While such constraints increase the complexity of the formulation, they remain compatible with standard optimal control techniques. The imposed nonnegativity of the initial conditions ensures that all trajectories remain within the biologically meaningful region
\[
\Bigl\{ (E,M) \;\big|\;
E_j(t) \ge 0,\; M_j(t) \ge 0,\;
\text{and the dynamics~\eqref{control dynamics} are satisfied}
\Bigr\}.
\]
As a consequence, every admissible solution of the optimal control problem is consistent with the fundamental biological requirement that molecular concentrations remain nonnegative. Conceptually, this dynamic formulation is closely related to the optimal control framework proposed in~\cite{Giordano2016}, where the objective is to maximize the accumulation of a target quantity over a finite time horizon. The key difference lies in the structure of the dynamics, since the explicit incorporation of macromolecular turnover introduces additional dissipation terms, leading to a more realistic but also more intricate model. Importantly, the present formulation provides a natural bridge between dynamic optimization and classical steady-state RBA. In particular, steady-state solutions of the optimal control problem are expected to recover standard RBA predictions, while transient regimes reveal how cells dynamically adjust their allocation strategies to approach optimal growth. The dynamic optimal control problem is now fully specified. In the next section, we derive first-order necessary optimality conditions using Pontryagin's Maximum Principle.

\subsection{Pontryagin’s Maximum Principle for dynamic resource allocation}

In order to characterize optimal cellular resource allocation strategies, we now analyze the optimal control problem introduced in the previous section using PMP.
This framework provides first-order necessary conditions
for optimality and allows us to derive structural properties
of optimal allocation strategies. More precisely, applying PMP yields necessary conditions for optimality in the form of an adjoint system coupled to the state dynamics. Owing to the affine dependence of the Hamiltonian on the control variable, optimal controls are characterized by a bang–bang structure, determined by the sign of an associated switching function. This structure admits a natural biological interpretation, whereby the cell alternates between distinct resource allocation regimes depending on the relative marginal benefits of enzyme and machinery production.

We recall that the state variables of the system are given by
the concentrations of enzymes and macromolecular machines,
denoted collectively by
\[
x(t) := \bigl(E(t), M(t)\bigr) \in \mathbb R^{N_m + N_p}_{\ge 0},
\]
whose dynamics are governed by
\begin{equation*}\label{eq:state_dynamics_PMP}
\dot x(t) = h\bigl(x(t), \alpha(t)\bigr),
\end{equation*}
where $\alpha(t) \in [0,1]$ represents the fraction of translational capacity allocated to enzyme synthesis. Throughout this section, we denote by $x(\cdot)$ the dynamic state of the system. This notation is deliberately distinguished from that of Section~\ref{Pre-Result-on-RBA}, where the static RBA framework is formulated in terms of the variable $Y=(E,M)$. The objective functional is
\begin{equation*}\label{eq:cost_functional_PMP}
\mathcal J(\alpha)
=
\int_0^T \mu\bigl(x(t)\bigr)\,\diff t,
\end{equation*}
where $\mu(x)$ denotes the instantaneous growth rate,
assumed to be a sufficiently smooth function of the cellular state.
The initial condition $x(0)=x_0$ is fixed, and there is no terminal constraint on $x(T)$. The admissible controls belong to the set
\[
\mathcal U
=
\bigl\{ \alpha \in L^\infty(0,T) \mid 0 \le \alpha(t) \le 1 \ \text{a.e.} \bigr\}.
\]
To apply PMP, we introduce the adjoint (costate) variables
\[
\eta(t) := (\eta_{\mathbb E(t)}, \eta_{\mathbb M(t)}) \in \mathbb R^{N_m + N_p},
\]
and define the Hamiltonian function
\begin{equation*}\label{eq:Hamiltonian}
\mathcal H(x, \eta, \eta^0, \alpha)
=
\eta^0\mu(x)
+
\eta \cdot h(x,\alpha).
\end{equation*}
Using the explicit form of the dynamics~\eqref{control dynamics},
the Hamiltonian can be written as
\begin{align*}
\mathcal H\bigl(x, \eta, \eta^0, \alpha\bigr)
&=
\eta^0\mu(E,M) + \sum_{j} \eta_{\mathbb E_j} \Bigl(\alpha \,\nu_{\mathbb E_j} -
(\mu+\gamma_{\mathbb Y_j}) E_j \Bigr) \nonumber\\
&\quad + \sum_{j} \eta_{\mathbb M_j} \Bigl( (1-\alpha)\,\nu_{\mathbb M_j} - (\mu+\gamma_{\mathbb Y_j}) M_j \Bigr).
\end{align*}
This expression highlights the linear dependence of the Hamiltonian
on the control variable $\alpha$, a key feature that will have important consequences for the structure of optimal controls.

Existence of an optimal solution for the class of control problems considered here follows from standard results in optimal control theory. Since no terminal state constraints are imposed, there are no controllability issues associated with reaching a prescribed final state. Moreover, the system dynamics are affine with respect to the control variable, which takes values in a compact and convex set, namely the closed interval $[0,1]$. In addition, one readily verifies that all trajectories remain bounded over any finite time horizon. Under these conditions, existence of an optimal control is ensured by Filippov’s theorem (see, e.g.\@~\cite{AgrachevSachkov2013}). For the optimal control problem~\eqref{OCP} with state variable $x \in \mathbb R^n$, PMP further guarantees the existence of a scalar $\eta^0 \le 0$ and a piecewise absolutely continuous adjoint mapping $\eta \colon [0,T] \to \mathbb R^n$, with $(\eta(\cdot), \eta^0) \neq (0,0)$, such that the extremal quadruple $(x, \eta, \eta^0, \alpha^\ast)$ satisfies the generalized Hamiltonian system
\[
\left\{
\begin{aligned}\label{OCP-hamilton}
    & \dot x(t) = \frac{\partial}{\partial \eta} \mathcal H(x , \eta , \eta^0, \alpha^\ast) \\
    & \dot \eta(t) = -\frac{\partial}{\partial x} \mathcal H(x , \eta , \eta^0, \alpha^\ast) \\ 
    & \mathcal H(x , \eta , \eta^0, \alpha^\ast) = \max_{\alpha \in [0, 1]}\mathcal H(x , \eta , \eta^0, \alpha),
\end{aligned}
\right.
\]
In other words, PMP states that if $\alpha^\ast$ is an optimal control with associated trajectory $x$, then the state evolves according to \[
\dot x(t) = h\bigl(x(t), \alpha^\ast(t)\bigr),
\qquad x(0)=x_0,
\]
and the costate variables satisfy
\begin{equation*}\label{eq:adjoint_equation}
\dot \eta(t)
=
-
\frac{\partial}{\partial x} \mathcal H
\bigl(x(t), \eta(t), \alpha^\ast(t)\bigr),
\end{equation*}
with terminal condition
\[
\eta(T)=0,
\]
since the final state is free and there is no terminal cost.
Notice that here, we use the fact that growth dilution affects all molecular species proportionally to the instantaneous growth rate $\mu (x)$, and explicitly, for each enzyme $E_j$ and macromolecular machine $M_j$, the adjoint equations are given by
\begin{align*}
\dot \eta_{\mathbb E_j}(t)
&=
- \eta^0
\frac{\partial \mu}{\partial E_j}
+
\eta_{\mathbb E_j}(\mu+\gamma_{\mathbb Y_j})
+
\eta_{\mathbb E_j} E_j \frac{\partial \mu}{\partial E_j}
+
\eta_{\mathbb M_j} M_j\frac{\partial \mu}{\partial E_j},
\\
\dot\eta_{\mathbb M_j}(t)
&=
-  \eta^0
\frac{\partial \mu}{\partial M_j}
+
\eta_{\mathbb M_j}(\mu+\gamma_{\mathbb Y_j})
+
\eta_{\mathbb M_j} M_j\frac{\partial \mu}{\partial M_j}
+
\eta_{\mathbb E_j} E_j \frac{\partial \mu}{\partial M_j}.
\end{align*}
These equations reflect the marginal value of enzymes
and macromolecular machines with respect to future growth.
Since the Hamiltonian is affine with respect to $\alpha$,
it can be written as $ \mathcal H(x,\eta, \eta^0,\alpha) = \mathcal H_0 + \alpha \mathcal H_1$. More precisely, 
\[
\begin{aligned} \mathcal H(x,\eta, \eta^0,\alpha) = 
    & \underbrace{\eta^0\mu(E,M) - \sum_{j} \eta_{\mathbb E_j}(\mu+\gamma_{\mathbb Y_j}) E_j - \sum_{j} \eta_{\mathbb M_j}\bigl((\mu+\gamma_{\mathbb Y_j}) M_j - \nu_{\mathbb M_j}\bigr)}_{\mathcal H_0} \\
    & + \alpha \underbrace{\sum_j
    \Bigl(
    \eta_{\mathbb E_j}\nu_{\mathbb E_j}
    -
    \eta_{\mathbb M_j}\nu_{\mathbb M_j}
    \Bigr)}_{\mathcal H_1}.
\end{aligned}
\]

The maximum condition states that for almost every $t \in [0,T]$,
the optimal control $\alpha^\ast(t)$ maximizes the Hamiltonian, e.g.\@
$
\alpha^\ast(t)
\in
\arg\max_{\alpha \in [0,1]}
\mathcal H\bigl(x(t), \eta(t), \eta^0,  \alpha\bigr),
$
which yields 
\[
\alpha^\ast(t)
=
\begin{cases}
1, & \mathcal H_1(x(t),\eta(t)) > 0, \\
0, & \mathcal H_1(x(t),\eta(t)) < 0, \\
\alpha_{sing}(t), & \mathcal H_1(x(t),\eta(t)) = 0.
\end{cases}
\]
Thus, optimal controls are generically of \emph{bang--bang} type, with possible switching times determined by the zeros of $\mathcal H_1$. Intervals on which $\mathcal H_1 \equiv 0$ correspond to \emph{singular arcs}, where additional optimality conditions must be imposed.

Within the present linear–affine modeling framework, the bang--bang structure suggests that growth-optimal resource allocation strategies may involve alternating phases in which translational resources are predominantly allocated either to metabolic enzymes or to other macromolecular machinery. This interpretation, however, should be understood within the scope of the modeling assumptions. In more detailed biological models, nonlinear effects such as enzyme saturation, cooperative interactions, or diminishing returns may alter the Hamiltonian structure and lead to smoother allocation profiles. In the present model, the switching function $\mathcal H_1$ compares the marginal growth value of producing enzymes versus structural or regulatory proteins. Positive values of $\mathcal H_1$ indicate that increasing enzymatic capacity enhances future growth more efficiently, whereas negative values favor investment in cellular infrastructure.

Protein turnover plays a critical role in this balance.
Higher degradation rates increase the effective cost
of maintaining macromolecular machines, thereby shifting the switching surface and modifying the optimal allocation strategy. Within the present linear--affine optimal control framework, the PMP characterization suggests that optimal controls may involve sharp transitions between enzyme-dominated and machine-dominated allocation regimes. In affine optimal control systems, such switching structures may in some cases lead to high-frequency switching phenomena,
commonly referred to as \emph{chattering} (see, e.g.\@ \cite{BoarottoSigalotti2019, ZelikinBorisov1994}). However, in the present RBA setting, this possibility should be interpreted cautiously. Such behavior may partly reflect the simplified linear structure of the model and the absence of regulatory or energetic penalties on rapid reallocations.
In more realistic biological settings, nonlinear regulation,
delayed cellular responses, or metabolic smoothing effects
would likely regularize the control profiles and produce smoother allocation dynamics. 

In the limit of infinitely fast switching, the time-averaged dynamics converge to a steady allocation that satisfies classical RBA constraints. This provides a rigorous interpretation of RBA as the steady-state envelope of an underlying dynamic optimal control problem. Overall, Pontryagin's Maximum Principle provides a powerful framework to analyze optimal cellular resource allocation, bridging dynamic regulation, macromolecular turnover, and steady-state growth optimization. In particular, steady-state solutions of the optimal control problem correspond to allocations that maximize growth under static resource constraints, thereby recovering the classical RBA formulation. From this viewpoint, RBA emerges as the stationary limit of a dynamic optimization process, while the optimal control framework provides a natural explanation for how such allocations may be reached over time.

To illustrate the structure of the optimal control problem and the consequences of PMP, we consider a minimal toy model capturing the essential trade-off between metabolic capacity and cellular machinery. We consider a simplified cell characterized by two state variables: a pool of metabolic enzymes $E(t)$ and a pool of macromolecular machinery $M(t)$. Cell growth is assumed to be limited by the least abundant resource, leading to the growth law \begin{equation*}\label{eq:toy_growth}
\mu(E,M) = \min\{\kappa_{\mathbb E} E,\; \kappa_{\mathbb M} M\},
\end{equation*}
where $\kappa_{\mathbb E},\kappa_{\mathbb M} > 0$ are efficiency constants. The protein synthesis dynamics are governed by ribosome-driven fluxes $\nu_{\mathbb E}(t)$ and $\nu_{\mathbb M}(t)$, which are dynamically allocated through the control variable $\alpha(t)\in[0,1]$. The system dynamics are given by
\begin{align*}
\dot E(t) &= \alpha(t)\,\nu_{\mathbb E}(t)
- \bigl(\mu(E,M) + \gamma_{\mathbb E}\bigr)\,E(t),  \label{eq:toy_E}\\
\dot M(t) &= \bigl(1-\alpha(t)\bigr)\,\nu_{\mathbb M}(t)
- \bigl(\mu(E,M) + \gamma_{\mathbb M}\bigr)\,M(t). 
\end{align*}

Here, the fluxes $\nu_{\mathbb E}(t)$ and $\nu_{\mathbb M}(t)$
represent the effective ribosomal synthesis capacities
for enzymes and macromolecular machines, respectively,
and may depend implicitly on the current cellular state.
The Hamiltonian associated with this problem is affine in the control $\alpha$, and PMP predicts a bang--bang structure for optimal allocation strategies. In particular, when enzymatic capacity is growth-limiting (i.e.\ $\kappa_{\mathbb E} E < \kappa_{\mathbb M} M$), the optimal strategy consists in allocating all translational resources to enzyme synthesis ($\alpha^\ast=1$). Conversely, when macromolecular machinery is limiting, resources are optimally redirected toward machinery production ($\alpha^\ast=0$). As the system evolves, the state typically oscillates around the balanced-growth manifold
\[
\kappa_{\mathbb E} E = \kappa_{\mathbb M} M,
\]
which coincides with the steady-state resource balance condition
of classical RBA. In the long-time limit, rapid switching between enzyme- and machinery-dominated regimes may occur, yielding an effective averaged allocation that recovers the steady-state RBA solution. Despite its simplicity, this toy model reproduces key qualitative features of the full RBA framework, such as bang--bang allocation strategies, possible chattering behavior near balanced growth, and convergence toward steady-state resource allocation. It therefore provides an intuitive illustration of how classical RBA can be interpreted as the steady-state envelope of an underlying dynamic optimal control problem.

\section{Conclusion and perspectives}

The RBA has emerged in recent years as a powerful constraint-based framework for studying the allocation of intracellular resources under growth-optimal conditions. By explicitly accounting for enzymatic capacities, macromolecular synthesis costs, and global cellular constraints, RBA provides a mechanistic description of how cellular systems distribute their limited biosynthetic resources in order to sustain growth. However, most existing formulations have been primarily developed for prokaryotic organisms under steady-state assumptions, thereby neglecting several biological mechanisms that become essential when considering eukaryotic cells or dynamically changing environments.

The main objective of the present work was therefore to extend the classical RBA framework toward a more general mathematical formulation capable of incorporating both intracellular turnover mechanisms and temporal adaptation processes. To achieve this, we first revisited the standard prokaryotic RBA formulation and introduced a refined description of molecular machine dynamics by explicitly accounting for degradation and turnover effects. From a biological perspective, this extension allows the model to capture the continuous renewal of intracellular components, which plays a central role in protein homeostasis, stress adaptation, and long-term cellular viability. From a mathematical perspective, the introduction of turnover terms modifies the structure of the synthesis constraints while preserving the convexity properties that make RBA computationally tractable.

Building upon this first extension, we then proposed a generalized formulation adapted to eukaryotic cells. Unlike prokaryotic organisms, eukaryotic cells exhibit a highly compartmentalized intracellular organization, where molecular species must be synthesized, transported, and maintained across multiple organelles subject to distinct physical constraints. To account for this complexity, we introduced compartment-specific density constraints together with transport processes linking intracellular compartments. This formulation provides, to the best of our knowledge, one of the first mathematical extensions of RBA explicitly designed to capture the structural organization of eukaryotic cells within a rigorous optimization framework.

A central part of this work was devoted to the mathematical analysis of the resulting RBA feasibility problems. In particular, we established fundamental structural properties of the proposed formulation, including convexity of the feasible set, monotonicity of feasibility with respect to the growth rate, and existence of feasible solutions under suitable assumptions. These properties are of primary importance, not only from a theoretical viewpoint, but also for practical numerical applications, since they guarantee that the associated optimization problems remain compatible with efficient linear programming techniques, even for large-scale cellular models.

In order to move beyond the classical steady-state paradigm, we subsequently introduced a dynamic resource allocation model formulated as an optimal control problem. In this dynamic setting, the allocation of translational resources between metabolic enzymes and other macromolecular machines is represented by a time-dependent control variable. This formulation allows the framework to capture transient adaptation strategies, where cells continuously redistribute their biosynthetic resources in response to internal constraints and growth requirements. By coupling these allocation decisions with macromolecular turnover and growth dilution, the proposed model establishes a natural bridge between intracellular physiology and dynamic optimization. To characterize optimal allocation strategies, we derived first-order necessary conditions using Pontryagin's Maximum Principle. The resulting adjoint system provides a quantitative interpretation of the marginal value associated with intracellular resources, thereby offering new insights into how cellular components contribute to long-term growth performance. Beyond its theoretical interest, this optimal control formulation opens the possibility of studying transient metabolic adaptations, switching behaviors, or time-dependent trade-offs between immediate growth and long-term cellular robustness.

Despite these advances, the present framework remains intentionally simplified in several respects. In particular, the dynamic formulation considered here relies on a single scalar allocation variable, which provides a minimal representation of translational resource allocation. Real cellular systems,  however, involve multiple interacting regulatory layers, including transcriptional regulation, post-translational modifications, signaling pathways, and environmental sensing mechanisms, which are not explicitly modeled here. In addition, the instantaneous growth rate is represented through a smooth proxy rather than by solving the underlying RBA feasibility problem at each time step. While this approximation facilitates the mathematical analysis, it may overlook certain nonlinear effects arising in genome-scale models. Furthermore, stochastic fluctuations, cell-to-cell variability, and explicit extracellular nutrient dynamics have not been considered in the present study.

Moreover, in this modeling framework, the metabolic, macromolecular and density constraints are represented by linear equalities and inequalities. This structural assumption plays a central role in many of the theoretical properties established in this manuscript, including convexity of the feasible sets, monotonicity with respect to the growth rate, and the tractability of the associated optimization problems. Likewise, the bang--bang structure of the optimal controls derived from PMP is closely linked to the affine dependence of the dynamics on the control variables. From a biological perspective, however, these linear relationships should be regarded as idealized approximations. Real cellular systems may exhibit nonlinear effects, including enzyme saturation, allosteric regulation, cooperative interactions, signaling feedback, and environment-dependent regulatory mechanisms, which are not explicitly captured in the present formulation. Consequently, the conclusions established here should be interpreted within the scope of the RBA assumptions rather than as universal properties of biological resource allocation. Nevertheless, the proposed framework provides a mathematically tractable foundation for analyzing dynamic resource allocation at genome scale, and offers a rigorous starting point for future extensions incorporating nonlinear regulatory mechanisms, stochastic effects, or experimentally calibrated kinetic constraints.

These limitations naturally suggest several promising directions for future research. A first extension would consist in introducing multiple control variables associated with distinct biosynthetic processes, such as transcription, translation, protein folding, or intracellular transport. Another important direction would be to couple the intracellular dynamics with extracellular nutrient availability, thereby allowing the study of cellular adaptation in fluctuating environments. From a computational perspective, the integration of the proposed framework with genome-scale metabolic reconstructions and efficient numerical optimal control methods constitutes a particularly promising avenue. More generally, the combination of resource allocation models, compartmentalized cellular physiology, and dynamic optimization may provide a powerful mathematical framework for studying cellular adaptation, metabolic engineering, and synthetic biology applications.

Overall, the present work provides a rigorous mathematical bridge between steady-state resource allocation models and dynamic optimal control approaches. By combining constraint-based modeling, compartmental cellular organization, macromolecular turnover, and dynamic optimization, it contributes to a deeper quantitative understanding of how living cells allocate their limited resources to sustain growth, maintain homeostasis, and adapt to changing environments. We hope that this framework will serve as a foundation for future theoretical, computational, and experimental investigations at the interface between mathematical biology, systems biology, and optimal control.

\addtocontents{toc}{\protect\setcounter{tocdepth}{-1}}
\section*{Acknowledgments}
\addtocontents{toc}{\protect\setcounter{tocdepth}{2}}

The author would like to acknowledge Vincent Fromion for valuable discussions on the RBA framework.

\bibliographystyle{abbrv}
\bibliography{bib}

\end{document}